\documentclass{article}[11pt]
\usepackage{amssymb, amsthm, amsmath}
\usepackage[a4paper]{geometry}
\usepackage{amssymb,latexsym,amscd}
\usepackage{epsfig}
\usepackage[active]{srcltx}
\usepackage{psfrag}
\usepackage{graphics,graphicx}
\usepackage{pstricks,pst-node,pst-tree}
\usepackage{braket}
\usepackage{color}

\usepackage{bm}
\usepackage{hyperref}

\definecolor{r}{rgb}{1,0,0}
\definecolor{b}{rgb}{0,0,1}

 \newcommand{\tr}[1]{}

\newcommand{\lie}[2]{\left[#1,#2\right]}

\newcommand{\beqn}{\begin{eqnarray}\begin{aligned}}
\newcommand{\eqn}{\end{aligned}\end{eqnarray}}
\newcommand{\id}{\mathbf{I}}

\theoremstyle{definition}
\newtheorem{thm}{Theorem}[section]

\newtheorem{res}{Result}

\newtheorem{lem}[thm]{Lemma}

\theoremstyle{definition}
\newtheorem{mydef}[thm]{Definition}

\title{Lie-Markov models derived from finite semigroups}
\author{Jeremy G. Sumner and Michael D. Woodhams}

\begin{document}

\maketitle

\begin{center}
\footnotesize{School of Physical Sciences, University of Tasmania, Australia}\\
\end{center}

\begin{abstract}
\noindent
We present and explore a general method for deriving a Lie-Markov model from a finite semigroup.
If the degree of the semigroup is $k$, the resulting model is a continuous-time Markov chain on $k$ states and, as a consequence of the product rule in the semigroup, satisfies the property of multiplicative closure.
This means that the product of any two probability substitution matrices taken from the model produces another substitution matrix also in the model.
We show that our construction is a natural generalization of the concept of group-based models. 
\end{abstract}

\section{Introduction} 

Recent work has defined and explored `Lie-Markov models' in the context of phylogenetic modelling \cite{sumner2012lie}.
These are the class of continuous-time Markov models which have the pleasing property of producing substitution matrices which are closed under matrix multiplication \cite{sumner2017multiplicative}.
We have argued this is an important consistency property for a phylogenetic model to possess \cite{sumner2012general} and have shown there is some evidence these models perform better than entrenched standard models on real data sets \cite{woodhams2015new}.

If one considers the definition of Lie-Markov models to be mathematically compelling in its own right, it is an interesting mathematical question to produce a complete enumeration of these models. 
Given the well-developed nature of the associated Lie group theory (see \cite{hall2015lie,stillwell2008naive} for excellent introductions), one might think it is simply a matter of looking in a standard reference to find a list of all Lie matrix groups and reject those which are not Markovian (in the appropriate sense).
However, this approach fails at the very first step.
The problem is that, from the abstract point of view, the classification of Lie groups is performed up to some reasonable isomorphism conditions.
For instance, the classic results of Cartan and Levi provide classifications of the associated Lie algebras via semi-simplicity and root systems.
The obstruction to using this classification to find Lie-Markov models is that this theory is always explored up to similarity transformations.
For applications to Markov chains it is crucial that (i) the Lie algebra is given as a concrete set of matrices (that is, a particular representation is specified), and (ii) these matrices are `stochastic' in the appropriate sense (to be defined below).
Both of these conditions renders the classical approach inappropriate.

In the most abstract setting, two Lie algebras are the isomorphic if there is a linear map between them that respects the Lie bracket multiplication.
However, two Lie-Markov models which are isomorphic as Lie algebras may be very different as Markov models; indeed they need not even be both defined on a state space of the same size.
We give a simple example of this in Sec~\ref{subsec:isodifferent}.
Thus in applications, we need to be able to distinguish Lie-Markov models that abstractly may form the same Lie algebra.
In fact, for a Lie-Markov model, we are less interested in which specific Lie algebra we are dealing with, than the fact that the model \emph{does}  form a Lie algebra (c.f. \cite{sumner2017multiplicative}).

A second approach might be to use the classification theorems for matrix groups and derive Lie-Markov models form these groups.
For example, the hierarchy of classical groups --- the orthogonal, unitary, and symplectic groups.
However, this approach also fails because it is not at all clear when a Lie-Markov model may be lurking in a standard case if only we could find the appropriate similarity transformation which would bring the matrices into stochastic form.

Given that the general Markov model forms a Lie algebra \cite{johnson1985markov}, a third approach may be to proceed by taking arbitrary constraints on this general case and rejecting those that do not form Lie algebras. 
However this approach also fails, since, as was explained in \cite{sumner2012lie}, there exist simple examples of infinite families of Lie-Markov models.
Since the constraints distinguishing them are somewhat arbitrary, the majority of the models in a given infinite family are uninteresting for practical applications, and, ideally, one would like a systematic scheme for picking out interesting cases from the parametrized family of possibilities.
 
In this vein, the approach developed in \cite{sumner2012lie}  rests upon the natural observation that the phylogenetic models that are used in practice have symmetries under nucleotide permutations.
For example, the general time-reversible (GTR) model, which is shown not to be multiplicatively closed in \cite{sumner2012lie}, nonetheless has complete symmetry under necleotide symmetries.
In \cite{sumner2012lie}, we presented a general method which produces all Lie-Markov models which have the symmetries of a given permutation group.
This method was then applied to show there are five Lie-Markov models with complete symmetry.
We also used this approach in \cite{fernandez2015lie} to show there are (approximately) 35 Lie-Markov models with symmetry that respects the partitioning of nucleotides in purines $AG$ and pyrimidines $CT$ (the precise count depends on inclusion/exclusion of some special cases depending upon one's preferences).
Currently, this is most powerful method we know of for systematically deriving Lie-Markov models.

However, an unappealing mathematical feature of this approach is described as follows.
The procedure fixates on the vector-space property of the Lie algebras to reduce the general case into a set of subspaces which individually respect the chosen group of nucleotide symmetries.
A potential Lie-Markov model is then built by taking sums of these subspaces and checking whether the resulting model happens to form a Lie algebra, or not.
If it does, we have a Lie-Markov model, if it doesn't, we forget the proposed model and proceed to the next option.
For the algebraically minded, this is a slightly unsatisfactory state of affairs since there is nothing in this approach which exploits the deep nature of the Lie group associated with each Lie-Markov model: each model either forms a Lie algebra or is doesn't, and there is no general principle guiding us to expect the answer to this question to go either way in an individual case.

It is also interesting to recall that the complete group of nucleotide permutation symmetries imposed in \cite{sumner2012lie}, as well as the more restricted `purine-pyrimidine' symmetries imposed in \cite{fernandez2015lie}, are strong enough that, coincidently, these models form Lie algebras by implication of the stronger property that they form matrix algebras --- we will explain the precise meaning of this in Sec~\ref{sec1} and provide an example illustrating this is not the case for all Lie-Markov models in Sec~\ref{subsec:notalg}. 
Thus one is left wondering whether a richer set of models is perhaps missing from what has been achieved so far.

We recall that the literature also contains two well-known classes of phylogenetic models that form matrix algebras: the so-called `group-based' \cite{semple2003phylogenetics} and `equivariant' \cite{draisma2009ideals} models (see  Sec~\ref{subsec:groupbased} and Sec~\ref{subsec:equivarient}, respectively).
The equivariant models impose an even stronger version of symmetry conditions than has been applied in our work on Lie-Markov models.
The strong symmetries of the equivariant models ensure they are matrix algebras and hence form Lie algebras. 
For this reason, the equivariant models occur naturally in the broader class of Lie-Markov models we have derived through our more general methods. 

Comparatively, the inspiration for the present paper can be seen as a generalization of the group-based models to `semigroup-based' models.
We will show that this natural generalization provides yet another means for deriving Lie-Markov models and we give a complete classification of all semigroup-based models in the cases of binary, three and four (DNA) state models.

\section{Rate matrices, stochasticity and Lie-Markov models}
\label{sec1}

We consider homogeneous continuous-time Markov chains with state space $X\!=\!\{1,2,\ldots,k\}$.
A \emph{rate-matrix} $Q$ is a $k\!\times\!k$ matrix with non-negative off-diagonal entries and zero column-sums.
Given a rate-matrix $Q$, we may compute the corresponding probability transition matrix using the exponential map:
\[
M(t)=e^{Qt},
\]
where the $ji$ entry of $M(t)$ is the conditional probability $\mathbb{P}[i\rightarrow j,\text{ in time }t ]$.


For our purposes, a \emph{model} is, for fixed $k$, a choice of a restricted class of rate-matrices $Q$.
Taking the case of the DNA state space $X\!=\!\{A,G,C,T\}\!\equiv\! \{1,2,3,4\}$ we have $k\!=\!4$.
The ever-popular general time-reversible model GTR \cite{tavare1986} is parametrized by the stationary distribution $\pi=(\pi_i)_{i \in X}$ of the chain  and `relative' rates $s_1,s_2,\ldots,s_6$:
\[
Q_{\text{GTR}}=
\left(
\begin{matrix}
\ast & \pi_1s_1 & \pi_1s_2 & \pi_1s_3 \\
\pi_2s_1 & \ast  & \pi_2s_4 & \pi_2s_5 \\
\pi_3s_2 & \pi_3s_4 & \ast   & \pi_3s_6 \\
\pi_4s_3 & \pi_4s_5 & \pi_4s_6 & \ast,
\end{matrix}
\right),
\]
(where the rows and columns are ordered according to $A,G,C,T$ and the missing entries $\ast$ are determined by the zero column-sum condition). 
The defining feature of this model is that it produces a Markov process that, at equilibrium, is identical to its time-reversed process.
This property is exhibited by observing the matrix entries of this model satisfy the detailed balance conditions $q_{ij}\pi_j=q_{ji}\pi_i$ for all choices $i\neq j$. 
The interpretation is that the rate of a transition $i\rightarrow j$ weighted by probability $\pi_i$ of being in state $i$ is equal to converse rate $j\rightarrow i$ weighted by probability $\pi_j$. 

Another popular phylogenetic model is the HKY model \cite{hasegawa1985} which has the additional constraints $s_1\!=\!s_6\!=\!\kappa$ and $s_2\!=\!s_3\!=\!s_4\!=\!s_5\!=\!1$, giving:
\[
Q_{\text{HKY}}=
\left(
\begin{matrix}
\ast & \pi_1\kappa & \pi_1 & \pi_1 \\
\pi_2\kappa & \ast  & \pi_2 & \pi_2 \\
\pi_3 & \pi_3 & \ast   & \pi_3\kappa \\
\pi_4 & \pi_4 & \pi_4\kappa & \ast.
\end{matrix}
\right).
\]
This model is again time-reversible and is motivated by distinguishing \emph{transitions} (substitutions within purines and pyrimdines, i.e. $A\leftrightarrow G,C\leftrightarrow T$) from \emph{transversions} (substitutions between purines and pyrimidines).

Below we also consider the symmetric model SYM which is obtained by taking $\pi$ to be the uniform distribution, and rescaling so that:
\[
Q_{\text{SYM}}=
\left(
\begin{matrix}
\ast & s_1 & s_2 & s_3 \\
s_1 & \ast  & s_4 & s_5 \\
s_2 & s_4 & \ast   & s_6 \\
s_3 & s_5 & s_6 & \ast.
\end{matrix}
\right).
\]


Roughly speaking, we say that a model is \emph{multiplicatively closed} if, for any two substitution matrices $M_1$ and $M_2$ arising from the model, the matrix product $M_1M_2$ is obtainable from the model.
We will see below that the GTR, HKY, and SYM models are not multiplicatively closed.

An example of a well-known multiplicatively closed model is the sub-model of GTR obtained by setting each relative rate equal to unity (the Felsenstein 81 \cite{felsenstein1981evolutionary} or `equal-input'  \cite{steel2016phylogeny} model):
\[
Q_{\text{F81}}=
\left(
\begin{matrix}
\ast & \pi_1 & \pi_1& \pi_1 \\
\pi_2 & \ast  & \pi_2 & \pi_2 \\
\pi_3 & \pi_3 & \ast   & \pi_3 \\
\pi_4 & \pi_4 & \pi_4 & \ast.
\end{matrix}
\right).
\]

As was discussed in \cite{sumner2012lie}, a sufficient condition for a multiplicatively closed model is that the set of rate-matrices forms a Lie algebra.
Under further reasonable assumptions regarding the construction of a model (essentially the model should be an intersection with an algebraic variety), the Lie algebra condition was more recently shown in \cite{sumner2017multiplicative} to also be necessary.
To state this result precisely, we follow the notation and definitions given in \cite{sumner2017multiplicative}.

For fixed $k$, let $\mathcal{L}^+$ denote the set of $k\times k$ rate-matrices $Q$.
That is, $Q\in \mathcal{L}^+$ if $Q$ has real, non-negative off diagonal entries and zero column-sums.  
We then let $\mathcal{L}\supset \mathcal{L}^+$ denote the set of real $k\times k$ matrices with zero column-sums but with the non-negativity condition relaxed.
A model $\mathcal{R}^+\subseteq \mathcal{L}^+$ is assumed to be expressible as an intersection $\mathcal{R}^+=\mathcal{R}\cap \mathcal{L}^+$ where $\mathcal{R}\subseteq \mathcal{L}$ is an algebraic variety.
This means that $\mathcal{R}$ is defined by some polynomial constraints:
\[
\mathcal{R}=\{Q\in \mathcal{L}:0\!=\!f_1(Q)\!=\!f_2(Q)\!=\ldots =\!f_r(Q)\},
\]
where each $f_i(Q)$ is a polynomial in the entries of the matrix $Q$.
Further, we also assume that $\mathcal{R}$ is the \emph{minimal} algebraic variety satisfying $\mathcal{R}^+=\mathcal{R}\cap \mathcal{L}^+$.

We recall:
\begin{mydef}
\label{def:liealg}
A (matrix) \emph{Lie algebra} is a set $\mathfrak{L}$ of matrices satisfying, for all $A,B\in \mathfrak{L}$ and scalars $\lambda\in \mathbb{R}$, the two conditions:
\begin{enumerate}
\item[(L1)] $A+\lambda B\in \mathfrak{L}$;
\item[(L2)] $\lie{A}{B}:=AB-BA\in \mathfrak{L}$.
\end{enumerate}
\end{mydef}
The first condition states that $\mathfrak{L}$ forms a real vector space under sums and scalar multiplication.
The operation $\lie{A}{B}$ is referred to as the `Lie bracket' or `commutator' and should be thought of as the natural product in the Lie algebra $\mathfrak{L}$.
(In the abstract formulation, a third condition known as the Jacobi identity is not needed here since we are restricting attention to \emph{matrix} Lie algebras only and this condition is automatic in this case.)

\begin{mydef}[\cite{sumner2017multiplicative}]
A model $\mathcal{R}^+=\mathcal{R}\cap \mathcal{L}^+$ is said to be \emph{multiplicative closed} if, for all choices $Q_1,Q_2\in \mathcal{R}^+$ and $t_1,t_2\geq 0$, we have:
\[
\textstyle{\frac{1}{t_1+t_2}}\log(e^{Q_1t_1}e^{Q_2t_2})\in \mathcal{R},
\]
where $\log$ denotes the standard power series for the matrix logarithm. 
\end{mydef}
As a consequence, we have: 
\[
\widehat{Q}:=\textstyle{\frac{1}{t_1+t_2}}\log(e^{Q_1t_1}e^{Q_2t_2})\quad  \implies \quad 
e^{\widehat{Q}(t_1+t_2)}=e^{Q_1t_1}e^{Q_2t_2}.
\] 
Thus, if $M_1\!=\!e^{Q_1t_1}$ and $M_2\!=\!e^{Q_2t_2}$ are derivable from the model, then so is their product $M_1M_2$ (via the rate-matrix $\widehat{Q}$).

One should note that we intentionally do not insist on the stronger condition that $\widehat{Q}\in \mathcal{R}^+$, since there are cases where $\log(e^{Q_1t_1}e^{Q_2t_2})$ has non-negative off-diagonal entries.
The definition is designed so that taking products of substitution matrices and then the matrix $\log$ does not, from the geometric point of view, produce matrices outside of $\mathcal{R}$, which is, reassuringly,  assumed to be minimal.

It follows that:
\begin{thm}[\cite{sumner2017multiplicative}]
A model $\mathcal{R}^+=\mathcal{R}\cap \mathcal{L}^+$ is multiplicative closed if and only if $\mathcal{R}$ forms a Lie algebra.
\end{thm}
For this reason, we refer to $\mathcal{R}^+$ as a Lie-Markov model whenever it is multiplicatively closed.
The most immediate consequence of this result is that a multiplicatively closed model $\mathcal{R}^+$ must be determined by \emph{linear} polynomial constraints $f_i(Q)$, that is, $\mathcal{R}$ is a linear space.

Comparative to the definition of a Lie algebra, consider:
\begin{mydef}
\label{def:matalg}
A \emph{matrix algebra} is a set $\mathcal{A}$ of matrices satisfying, for all $A,B\in \mathcal{A}$ and scalars $\lambda$, the two conditions:
\begin{enumerate}
\item $A+\lambda B\in \mathcal{A}$;
\item $AB\in \mathcal{A}$.
\end{enumerate}
\end{mydef}

As a simple consequence of these definitions:
\begin{lem}
\label{lem:algLie}
Any matrix algebra $\mathcal{A}$ forms a (matrix) Lie algebra under commutators.
\end{lem}
\begin{proof}
We need only check that, for all $A,B\in\mathcal{A}$, we have $\lie{A}{B}\in \mathcal{A}$.
But this follows easily since $AB\in \mathcal{A}$, $BA\in\mathcal{A}$, and $\mathcal{A}$ is closed under summation.
\end{proof}
There are certainly examples of (matrix) Lie algebras that do not form matrix algebras, so the converse is false in general.
We give an example of a Lie-Markov model that does not form a matrix algebra in Sec~\ref{subsec:notalg}.

The GTR and HKY models are not Lie-Markov models since they are implicitly defined by non-linear constraints on their matrix entries and hence fail to satisfy (L1).
For the GTR model these constraints are known to be cubic by Kolomogorov's criterion for detailed balance \cite{kolomogorov1936}.
Similarly, the matrix entries of the HKY model satisfy the constraint $q_{12}q_{23}=q_{21}q_{13}$, and this constraint is not implied by simpler, linear conditions. 
On the other hand the SYM and F81 models clearly satisfy (L1), but only F81 satisfies condition (L2) as well.

To see that the SYM model fails (L2), consider two rate-matrices $Q_1,Q_2\in \text{SYM}$ so $Q_i\!=\!Q_i^T$. 
Also assume $\lie{Q_1}{Q_2} \neq 0$ (such examples certainly exist).
Now consider the transpose of the commutator:
\[
\lie{Q_1}{Q_2}^T=\left(Q_1Q_2-Q_2Q_1\right)^T=Q_2^TQ_1^T-Q_1^TQ_2^T=\lie{Q_2}{Q_1}=-\lie{Q_1}{Q_2},
\]
so $\lie{Q_1}{Q_2}$ is an anti-symmetric matrix and (L2) fails to hold for this model. 

To see that F81 model satisfies condition (L2), we exploit the fact that (L1) states that a Lie-algebra is a vector subspace of matrices and hence has a basis.
Confirming condition (L2) is then be achieved by checking the commutators of all pairs of basis elements.
We define the matrices $\{R_1,R_2,R_3,R_4\}$ via
\[
Q=\left(
\begin{matrix}
\ast & \pi_1 & \pi_1& \pi_1 \\
\pi_2 & \ast  & \pi_2 & \pi_2 \\
\pi_3 & \pi_3 & \ast   & \pi_3 \\
\pi_4 & \pi_4 & \pi_4 & \ast.
\end{matrix}
\right)=\pi_1R_1+\pi_2R_2+\pi_3R_3+\pi_4R_4\in \text{F81}.
\]
Taking $\mathcal{R}_{\text{F81}}=\text{span}_{\mathbb{R}}(R_1,R_2,R_3,R_4)$, explicit computation then shows
\begin{equation}
\label{eq:f81com}
\lie{R_i}{R_j}=R_i-R_j\in \mathcal{R}_{\text{F81}},
\end{equation}
so (L2) is satisfied, as required.
Notice here that  $\lie{R_i}{R_j}$ has some negative off-diagonal entries --- this is why we are required to expand the definition of rate-matrices to have entries from all of $\mathbb{R}$.
A major punchline for the approach we explore in this paper is that we will show how to derive the commutator relations (\ref{eq:f81com}) \emph{without the need to implement any matrix computations}.

In the applied setting, we of course only use stochastic rate-matrices.
Thus an additional feature of the theory is that, given a Lie-Markov model $\mathcal{R}^+$ is defined as the intersection $\mathcal{R}^+=\mathcal{R}\cap\mathcal{L}^+$, in general there are multiple approaches to parametrizing $\mathcal{R}^+$ as a subset of $\mathcal{L}$.
General tools for finding sensible parametrizations are discussed in \cite{fernandez2015lie} and \cite{woodhams2015new}.

\subsection{Model symmetries}

Suppose $\mathcal{R}^+\subset \mathcal{L}^+$ is a Markov model on $k$-states  (not necessarily multiplicatively closed) and $G\leq \mathcal{S}_k$ is a permutation group.
\begin{mydef}
\label{def:symm}
We say that $\mathcal{R}^+$ has $G$-symmetry if $G$ is the maximal permutation group such that, for all $\sigma\in G$, we have:
\[
Q\in \mathcal{R}^+ \implies K_\sigma QK_\sigma^T\in \mathcal{R}^+,
\]
where $K_\sigma$ is the standard $k\times k$ permutation matrix corresponding to $\sigma$.
\end{mydef}
In other words if, according to the permutation $\sigma\in G$, we simultaneously permute the rows and columns of a rate matrix in the model we obtain another rate matrix also in the model.

It is not hard to see that the GTR, SYM, and F81 models have $\mathcal{S}_4$ symmetry, whereas the HKY model has reduced symmetries given by the dihedral group:
\[
D_4=\{e,(12),(34),(12)(34),(13)(24),(14)(23),(1324),(1423)\}.
\] 

Below, in Sec~\ref{subsec:equivarient}, we will discuss the stronger notion of model symmetry  used to define the equivariant models.

We also use this notion to define model equivalence:
\begin{mydef}
Two models 
$\mathcal{R}_1^+,\mathcal{R}_2^+\in \mathcal{L}^+$ on $k$-states are \emph{isomorphic} if there exists a permutation $\sigma\in \mathcal{S}_k$ such that:
\[
Q\in \mathcal{R}_1^+\Longleftrightarrow K_\sigma QK_\sigma^T\in \mathcal{R}_2^+
\]
\end{mydef}

Considering the HKY model, given that the eight permutations $\sigma\in D_4$ are the only symmetries of this model, we see that there should exists $4!/8=3$ isomorphic variants of this model.
These are easily understood as corresponding to the three possible partitionings of nucleotides into two set of two: $AG|CT$, $AT|CG$, and $AC|GT$.

In \cite{sumner2012lie}, we enumerated all the Lie-Markov DNA models with full symmetry $\mathcal{S}_4$ and followed this up in \cite{fernandez2015lie}, by enumerating all the Lie-Markov models with dihedral symmetry $D_4$.
In \cite{woodhams2015new} we then explored the performance of these models on real data sets taking note of the three hierarchies of models that, similarly to the HKY model, arise from the three choices of nucleotide partitions.

\section{Semigroup-based models}
\label{sec:semi}

We begin by showing how to interpret the F81 model as arising from an (abstract) semigroup of degree four.
We then show that this semigroup generalizes to degree $k$ and we similarly obtain the `equal-input' model \cite{steel2016phylogeny} on any number of states $k$. 
We then set up a general framework for deriving semigroup-based models and explain how this is a natural generalization of the notion of a group-based model.

Consider the semigroup $S\!=\!\{a_1,a_2,a_3,a_4\}$ with, for all $i,j\in\{1,2,3,4\}$, multiplication given by $a_ia_j\!=\!a_i$.
To confirm this is indeed semigroup, we need only check that multiplication is associative: $(a_ia_j)a_k\!=\!a_ia_k\!=\!a_i\!=\!a_ia_j\!=\!a_i(a_ja_k)$.

We may represent this semigroup using $4\times 4$ matrices $\{A_1,A_2,A_3,A_4\}$ defined by their action on standard unit vectors $e_i$ mimicking the multiplication $a_ia_j=a_i$  via $A_ie_j=e_i$.
From this we see that $A_i$ is the matrix with 1s on row $i$ and zeros elsewhere.
For example
\[
A_1=
\left(
\begin{matrix}
1 & 1 & 1 & 1 \\
0 & 0 & 0 & 0 \\
0 & 0 & 0 & 0 \\
0 & 0 & 0 & 0 
\end{matrix}
\right),
\]
and $A_iA_j\!=\!A_i$ generally.

We then construct the rate-matrices $R_i:=-\id+A_i\in \mathcal{L}^+$ and see that these are none other than the basis elements for the F81 model given in Section~\ref{sec1}, for example:
\[
R_1=
\left(
\begin{matrix}
0 & \phantom{-}1 & \phantom{-}1 & \phantom{-}1 \\
0 & -1 & \phantom{-}0 & \phantom{-}0 \\
0 & \phantom{-}0 & -1 & \phantom{-}0 \\
0 & \phantom{-}0 & \phantom{-}0 & -1 
\end{matrix}
\right)\in \mathcal{R}^+_{\text{F81}}.
\]
Further, as promised in the previous section, the commutators  follow immediately from the multiplication rule in the semigroup:
\[
\left[R_i,R_j\right]=\left[A_i,A_j\right]=A_iA_j-A_jA_i=A_i-A_j=R_i-R_j,
\]
where, in the first  equality, we have used the fact that commutators are blind to inclusion of scalar multiples of the identity matrix $\id$.

This example immediately generalizes to the degree $k$ semigroup $S=\{a_1,a_2,\ldots,a_k\}$ with multiplication rule $a_ia_j=a_i$.
The Markov model obtained is referred to as the `equal input' model \cite{steel2016phylogeny} with rate matrices:
\[
Q=
\alpha_1R_1+\alpha_2R_2+\ldots+\alpha_kR_k
=
\left(
\begin{matrix}
\ast & \alpha_1 & \alpha_1 & \ldots & \alpha_1 \\
\alpha_2 & \ast & \alpha_2 & \ldots & \alpha_2\\
\vdots\\ 
\alpha_k & \alpha_k & \alpha_k & \ldots &  \ast
\end{matrix}
\right),
\]
and commutators $\left[R_i,R_j\right]=R_i-R_j$.

This procedure generalizes to produce a $k$-state Lie-Markov model from any degree $k$ semigroup.
Our process of converting each semigroup element $a_i$ into a matrix is inspired from taking the `regular representation' of a group.
In detail, suppose $a_i,a_j,a_k\in S$ satisfy $a_k=a_ia_j$.
Then the $j$th column of the matrix $A_i$ has a single non-entry entry 1 in the $k$th row.
It is then clear that $A_k\!=\!A_iA_j$, mimicking the multiplication in the semigroup, and each $L_i\!:=\!-\id+A_i\in \mathcal{L}^+$ is a rate-matrix.
It is important to note however that the resulting map from the semigroup to the matrices $A_i$ is not necessarily injective (as it is for a group).

Given an enumeration of all semigroups of size $k$, our general procedure for producing semigroup-based models with $k$ states is then, for each semigroup $S=\{a_1,a_2,\ldots,a_k\}$:
\begin{enumerate}
\item List the set of $k\times k$ matrices $A_1,A_2,\ldots,A_k$ resulting from the regular representation of $S$;
\item Define the rate-matrices $L_i=-\id+A_i\in \mathcal{L}^+$;
\item Take $\mathcal{R}=\text{span}_\mathbb{R}\left(L_1,L_2,\ldots, L_k\right)$ and define the model $\mathcal{R}^+=\mathcal{R}\cap \mathcal{L}^+$.
\end{enumerate}

We discuss the connection of semigroup-based models to the usual construction of group-based models in the next section.
Presently, we note:
\begin{thm}
Every semigroup-based model is a Lie-Markov model.
\end{thm} 
\begin{proof}
(L1) is true by construction of $\mathcal{R}$ as a linear span, and (L2) follows by the multiplicative closure in the semigroup $S$.
Indeed, suppose $a_i,a_j\in S$ satisfy $a_ia_j=a_k$, then:
\[
L_iL_j=(-\id+A_i)(-\id+A_j)=\id-A_i-A_j+A_k=-L_i-L_j+L_k\in \mathcal{R},
\]
so, applying Lemma~\ref{lem:algLie}, we see that $\mathcal{R}$ forms a Lie algebra.
\end{proof}

The reader should note that, since the regular representation of a semigroup is not necessarily injective, the matrices $A_1,A_2,\ldots,A_k$ need not be distinct.
Further, as we will see in Sec~\ref{sec:antidiff}, it is also possible that some $A_i\!=\id$ and consequently $L_i\!=\!0$.
Thus, in general $\dim(\mathcal{R})\leq |S|$.
This means that there are examples of non-isomorphic semigroups that produce isomorphic (or even equal) Markov models.
An example of this is given in Sec~\ref{sec:twostate}.

We recall that two semigroups $S,S'$ are \emph{isomorphic} if there exists a bijection $\varphi:S\rightarrow S'$ satisfying, for all $s_1,s_2\in S$:
\[
\varphi(s_1s_2)=\varphi(s_1)\varphi(s_2).
\]
It should be clear that the Markov models produced by two isomorphic semigroups differ only up to a possible permutation of states.
Thus, in what follows, we need only consider non-isomorphic semigroups.

We also recall that two semigroups $S,S'$ are \emph{anti-isomorphic} if there exists a bijection $\overline{\varphi}:S\rightarrow S'$ satisfying, for all $s_1,s_2\in S$:
\[
\overline{\varphi}(s_1s_2)=\overline{\varphi}(s_2)\overline{\varphi}(s_1).
\]
An easy way to construct an anti-isomorphism is to take a semigroup $S$ and then define $S'$ by reversing the multiplication on $S$:
\[
s_1s_2=s_3\text{ in }S\quad\leftrightarrow\quad  s_2s_1=s_3 \text{ in }S'.
\]
In this case it is said that $S'$ is the anti-isomorphic copy of $S$.

When enumerating semigroups it is often the case that anti-isomorphic semigroups are not treated separately.
However, for our purposes, it is the case that the two semigroups can produce radically different Markov models.
We illustrate with the following example.

\subsection{Anti-isomorphic semigroups can produce different Lie-Markov models}
\label{sec:antidiff}
Consider the anti-isomophic copy of the semigroup underlying the F81 model above: $S'=\{a_1,a_2,a_3,a_4\}$ with multiplication rule $a_ia_j\!=\!a_j$.
The regular representation produces the matrices
\[
A_1=
\left(
\begin{matrix}
1 & 0 & 0 & 0 \\
0 & 1 & 0 & 0 \\
0 & 0 & 1 & 0 \\
0 & 0 & 0 & 1 
\end{matrix}
\right)=\id=A_2=A_3=A_4,
\]
which (beyond being non-injective) gives the trivial Lie-Markov model where all rate matrices are zero: $Q\!=\!0$.

Thus, we treat anti-isomorphic semigroups independently in what follows.

\subsection{Different Lie-Markov models can form isomorphic Lie algebras}
\label{subsec:isodifferent}
Consider the general $2$-state Markov model:
\[
Q=
\left(
\begin{matrix}
-\alpha & \phantom{-}\beta \\
\phantom{-}\alpha & -\beta 
\end{matrix}
\right)
=\alpha\left(
\begin{matrix}
-1 & 0 \\
\phantom{-}1 & 0 
\end{matrix}
\right)+\beta\left(
\begin{matrix}
0 & \phantom{-}1 \\
0 & -1 
\end{matrix}
\right)=\alpha L_1+\beta L_2.
\]
This is a Lie-Markov model since it forms a matrix algebra:
\[
L_1^2=-L_1,\quad L_1L_2=-L_2,\quad L_2L_1=-L_1,\quad L_2^2=-L_2,
\]
and hence, following Lemma~\ref{lem:algLie}, a Lie-algebra:
\begin{equation}
\label{eq:affine}
\lie{L_1}{L_2}=L_1-L_2.\nonumber
\end{equation}
The geometric consequences of the identification of the Lie algebra underlying this model are explored in \cite{sumner2013lie}.

We will not review the construction, but an interesting 3-state model arises from the $2$-state general  Markov model using the method given in \cite{jarvis2012markov}.
This model has rate matrices given by 
\[
Q'=
\left(
\begin{matrix}
-2\alpha & \beta & 0 \\
\phantom{-}2\alpha & -\alpha-\beta & \phantom{-}2\beta\\
0 & \alpha & -2\beta
\end{matrix}
\right)
=
\alpha\left(
\begin{matrix}
-2 & \phantom{-}0 & 0 \\
\phantom{-}2 & -1 & 0 \\
\phantom{-}0 & \phantom{-}1 & 0
\end{matrix}
\right)
+\beta
\left(
\begin{matrix}
0 & \phantom{-}1 & \phantom{-}0 \\
0 & -1 & \phantom{-}2 \\
0 & \phantom{-}0 & -2
\end{matrix}
\right)
= \alpha L'_1+\beta L'_2
\]
satisfying 
\[
\lie{L'_1}{L'_2}=L'_1-L'_2.
\]

Since these two models define two-dimensional Lie algebras satisfying the same commutator relations, they are isomorphic as Lie algebras.
In fact, for \emph{any} number of character states $k$, the method given in \cite{jarvis2012markov} produces a two-dimensional Lie-Markov model which is isomorphic, as a Lie algebra, to the 2-state general Markov model.
This illustrates that Lie algebra isomorphism is not (in itself) a useful tool for identifying distinct Lie-Markov models.

\subsection{Not all Lie-Markov models form matrix algebras}
\label{subsec:notalg}

For completeness of discussion, we give an example of a Lie-Markov model that forms a Lie algebra (Def~\ref{def:liealg}) without satisfying the stronger condition of forming a matrix algebra (Def~\ref{def:matalg}).
In fact, the 3-state example from the previous section is sufficient.
%

Consider the matrix product:
\[
L'_1L'_2=
\left(
\begin{matrix}
0 & -2 & \phantom{-}0 \\
0 & \phantom{-}3 & -2 \\
0 & -1 & \phantom{-}2
\end{matrix}
\right).
\]
This is clearly not expressible as a linear combination of $L_1'$ and $L_2'$.
Thus we see that there exists examples of Lie-Markov models which do \emph{not} form matrix algebras.

\subsection{Equivarient models}
\label{subsec:equivarient}

As mentioned in the introduction, an important class of Markov models that form matrix algebras (and hence are Lie-Markov models) are the `equivariant' models \cite{draisma2009ideals}.
These models were originally defined as sets of substitution matrices, but the idea is easily translated into the setting of rate matrices, as was described in \cite{sumner2012lie} and reproduced presently.

Fix a permutation group $G\leq \mathcal{S}_k$.
The equivariant model corresponding to $G$ is then obtained by taking the set of rate matrices $Q$ which are invariant under simultaneous row and column permutations by $\sigma\in G$.
Concretely, if $\sigma\in G$ is a permutation and $K_\sigma$ is the corresponding $k\times k$ permutation matrix, then the rate matrices $Q$ in the equivariant model have the defining feature
\begin{equation}
\label{eq:equivariant}
K_\sigma QK_\sigma^T=Q.\nonumber
\end{equation}

A simple consequence is that each equivariant model forms a matrix algebra via
\[
K_\sigma(Q_1Q_2)K_\sigma^T=(K_\sigma Q_1K_\sigma^T)(K_\sigma Q_2K_\sigma^T)=Q_1Q_2.
\]
Hence, following Lemma~\ref{lem:algLie}, the equivariant models form Lie algebras and are therefore Lie-Markov models.

For example, if we take the group of dihedral permutations 
\[
D_4=\{e,(12),(34),(12)(34),(13)(24),(14)(23),(1324),(1423)\}
\] 
(the symmetries of a square), we obtain the Kimura 2 parameter model \cite{kimura1980simple} as an equivariant model:
\[
Q_{\text{K2ST}}
=
\left(
\begin{matrix}
\ast & \alpha & \beta & \beta \\
\alpha & \ast & \beta & \beta\\
\beta & \beta & \ast & \alpha\\
\beta & \beta & \alpha & \ast
\end{matrix}
\right).
\]


\subsection{Connection to group-based models}
\label{subsec:groupbased}

Here we describe the usual construction of group-based models \cite{semple2003phylogenetics}.
We recall that the construction of a group-based model is usually thought of being valid only for \emph{abelian} (finite)  groups.  
However, following \cite{sumner2012lie}, we reinterpret the construction using the concept of the regular representation of a group and show how this allows us to construct models for general, possibly non-abelian, (finite) groups.
As in the case of the equivariant models, the construction can be implemented using a substitution or rate-matrix formulation and the results are equivalent.

Given a finite abelian group $G$, consider the state space $X\!=\!G$ and fix a linear function $f\!:\!G\rightarrow \mathbb{R}_{\geq 0}$ denoted as $f(g)\!=\!\alpha_g$ for all $g\in G$.
Using additive notation in $G$, for each pair $i,j\in X\!=\!G$ let the rate of substitution $i\rightarrow j$ be given by $f(i-j)$ where $i-j\in G$. 
Then we construct the rate-matrix $Q$ with off-diagonal entries $Q_{ij}=f(i-j)$ and diagonal entries determined by the zero sum condition.

Repeating what is done in \cite{sumner2012lie}, we show that this concept naturally extends to a general finite group $G$ (not necessarily abelian) by invoking the concept of the regular representation.
We recall that the regular representation of $G$ is given by mapping each $g\in G$ to a $|G|\times |G|$ (permutation) matrix $K(g)$ by setting the entry corresponding to each pair $g_1,g_2\in G$ equal to 1 if $g_1\!=\!gg_2$ and equal to 0 otherwise.
These matrices then satisfy the rule $K(g)K(g')=K(gg')$ and, as above, we may define the rate-matrices
\[
L_g=-\id+K(g),
\] 
which naturally form a matrix (and hence Lie) algebra:
\[
L_gL_{g'}=\id-K(g)-K(g')+K(gg')=-L_g-L_{g'}+L_{gg'}.
\]

If $G$ is abelian, it is not hard to show we obtain exactly the group-based model corresponding to $G$.
However, under this construction it is no longer necessary for $G$ to be abelian.

For example, consider the Klein 4-group:
\[
V_4=\{e,(12)(34),(13)(24),(14)(23)\}.
\]
This group produces the well-known Kimura 3ST model \cite{kimura1981estimation}:
\[
\left(
\begin{matrix}
\ast & \alpha & \beta & \delta \\
\alpha & \ast & \delta & \beta\\
\beta & \delta & \ast & \alpha\\
\delta & \beta & \alpha & \ast
\end{matrix}
\right).
\]

As the simplest non-abelian example, we set $G=\mathcal{S}_3$, the symmetric group on three elements, and obtain the six-state Markov model with rate matrices:
\[
Q=
\alpha_{(12)}L_{(12)}+\ldots +\alpha_{(132)}L_{(132)}
=
\left(
\begin{matrix}
\ast & \alpha_{(12)} & \alpha_{(132)} & \alpha_{(123)} & \alpha_{(23)} & \alpha_{(13)} \\  
\alpha_{(12)} & \ast & \alpha_{(23)} & \alpha_{(13)} & \alpha_{(132)} & \alpha_{(123)}  \\
\alpha_{(123)} & \alpha_{(23)} & \ast & \alpha_{(132)} & \alpha_{(13)} & \alpha_{(12)}\\ 
\alpha_{(132)} & \alpha_{(13)} & \alpha_{(123)} & \ast & \alpha_{(12)} & \alpha_{(23)}  \\
\alpha_{(23)} & \alpha_{(123)} & \alpha_{(13)} & \alpha_{(12)} & \ast & \alpha_{(132)}\\
\alpha_{(13)} & \alpha_{(132)} & \alpha_{(12)} & \alpha_{(23)} & \alpha_{(123)} & \ast
\end{matrix}
\right),
\]
and commutators
\[
\left[L_\sigma, L_{\sigma'}\right]=L_{\sigma\sigma'}-L_{\sigma'\sigma},
\]
for all choices $\sigma,\sigma'\in \mathcal{S}_3$.

Although this generalization to non-abelian groups is mathematically appealing, it is not of much use in phylogenetics with $k\!=\!4$ DNA states, since there are exactly two group of degree 4: $C_4$ and $C_2\times C_2$, both of which are abelian, and hence are already obtainable using the standard approach to group-based models. 
However, the astute reader will have noticed that there is nothing in the above that uses the availability of algebraic inverses in the group $G$.
Hence we may generalize immediately to semigroups $S$ and obtain exactly the construction given above for deriving semigroup-based models.

\section{Results}

\subsection{Four-state semigroup models}

In this section we present our results of exploring the semigroup-based models derived from degree-four semigroups.
There are 188 degree-four non-isomorphic semigroups with 126 of these being neither isomorphic or anti-isomorphic \cite{forsythe1955swac} (in other words there are 62 semigroups that occur as anti-isomorphic pairs). 
We implemented our procedure for deriving semigroup-based models to all 188 semigroups and then removed those models which are the same under permutation of nucleotide states.
From this process we produced 131 distinct Lie-Markov models.
However, most of these models are reducible and/or have absorbing states as Markov chains and hence are not of interest for our motivations in phylogenetic modelling.
All of these models are presented in the supplementary material.
To illustrate, examples of semigroup-based models with an absorbing state are given in the two- and three-state results below.

After removal of these uninteresting cases, we found:

\begin{res}
There are precisely \emph{four} distinct non-reducible, non-absorbing four-state Lie-Markov models derivable from degree-four semigroups.
Specially:
\begin{itemize}
\item The F81 model (discussed above);
\item The Kimura 3ST model (group-based model discussed above);
\item Model 3.3b (from the previously presented Lie-Markov model hierarchy \cite{fernandez2015lie});
\item A new model not previously observed.
\end{itemize}
\end{res}

\medskip
\noindent
\underline{Model 3.3b}:
We recall that the rate-matrices in Model 3.3b are expressible in the form
\[
Q_{3.3b}=
\left(
\begin{matrix}
\ast & \alpha & \beta & \gamma \\
\alpha & \ast & \gamma & \beta\\
\gamma & \beta & \ast & \alpha\\
\beta & \gamma & \alpha & \ast
\end{matrix}
\right),
\]
which can be understood as the `twisted' cousin to the Kimura 3 parameter model.
We were not previously aware that this is a semigroup-based model and it is amusing to `reverse-engineer' it to find it is based on the semigroup with multiplication table:
\[
\begin{tabular}{c|cccc}
    & $a_1$ & $a_2$ & $a_3$ & $a_4$  \\
    \hline
    $a_1$ & $a_2$ & $a_1$ & $a_4$ & $a_3$  \\
    $a_2$ & $a_4$ & $a_3$ & $a_1$ & $a_2$  \\
    $a_3$ & $a_3$ & $a_4$ & $a_2$ & $a_1$  \\
    $a_4$ & $a_1$ & $a_2$ & $a_3$ & $a_4$  
      \end{tabular}.
\]

This model has the symmetries of the dihedral group and hence respects the partitioning of nucleotides into purine and pyrimidines.
This also tells us that there are three distinct (but isomorphic) copies of this model corresponding to the three possible partitionings:  $AG|CT$, $AT|CG$ and $AC|GT$.

\medskip
\noindent
\underline{New semigroup-based model}:
On the other hand, the model previously unknown to us is based on the semigroup
\[
\begin{tabular}{c|cccc}
    & $a_1$ & $a_2$ & $a_3$ & $a_4$  \\
    \hline
    $a_1$ & $a_1$ & $a_1$ & $a_3$ & $a_3$  \\
    $a_2$ & $a_2$ & $a_2$ & $a_4$ & $a_4$  \\
    $a_3$ & $a_3$ & $a_3$ & $a_1$ & $a_1$  \\
    $a_4$ & $a_4$ & $a_4$ & $a_2$ & $a_2$  
      \end{tabular},
\]
and has rate matrices of the form
\[
Q=
\left(
\begin{matrix}
\ast & \alpha & \gamma & \gamma \\
\beta & \ast & \delta & \delta\\
\gamma & \gamma & \ast & \alpha\\
\delta & \delta & \beta & \ast
\end{matrix}
\right)
=
\alpha L_1+\beta L_2+\gamma L_3 +\delta L_4,
\]
with commutators
\[
\begin{tabular}{ccc}
$\lie{L_1}{L_2}=L_1-L_2$, & $\lie{L_1}{L_3}=0$, & $\lie{L_1}{L_4}=L_3-L_4$,\\
$\lie{L_2}{L_3}=L_4-L_3$,  & $\lie{L_2}{L_4}=0$, & $\lie{L_3}{L_4}=L_1-L_2$,
\end{tabular}
\]
(as can be read directly off the semigroup multiplication table).

Direct computation shows that the symmetry group  of this model is the Klein 4-group:
\[
V_4=\{e,(12)(34),(13)(24),(14)(23)\}.
\]
This means there are $6=4!/4$ distinct (but isomorphic) variants of this model to test on real data sets.

\subsection{Two-state semigroup models}
\label{sec:twostate}
In the case of $k\!=\!2$ binary states, there are precisely five non-isomorphic semigroups.
We denote each as $S=\{a_1,a_2\}$.
\medskip

\noindent
\underline{Semigroup 1}: 
\[
\begin{tabular}{c|cc}
    & $a_1$ & $a_2$ \\
    \hline
    $a_1$ & $a_1$ & $a_1$   \\
    $a_2$ & $a_1$ & $a_1$   
      \end{tabular}.
      \]
Taking the regular representation produces the matrices 
\[
A_1=
\left(
\begin{matrix}
1 & 1 \\
0 & 0
\end{matrix}
\right)
,\quad 
A_2=
\left(
\begin{matrix}
1 & 1 \\
0 & 0
\end{matrix}
\right).
\]
Thus the regular representation is not injective and produces the one-dimensional Lie-Markov model with an absorbing state and rate-matrices given by:
\[
Q=
\left(
\begin{matrix}
0 & \alpha \\
0 & -\alpha
\end{matrix}
\right).
\]

Since this model treats the two states very differently, it has trivial symmetry group $\{e\}<\mathcal{S}_2$.

\medskip
\noindent
\underline{Semigroup 2}: 
\[
\begin{tabular}{c|cc}
    & $a_1$ & $a_2$ \\
    \hline
    $a_1$ & $a_1$ & $a_1$   \\
    $a_2$ & $a_1$ & $a_2$   
      \end{tabular}
\]
Taking the regular representation produces the matrices 
\[
A_1=
\left(
\begin{matrix}
1 & 1 \\
0 & 0
\end{matrix}
\right)
,\quad 
A_2=
\left(
\begin{matrix}
1 & 0 \\
0 & 1
\end{matrix}
\right).
\]
This time the representation is injective, but, since $-\id+A_2\!=\!0$, this semigroup produces exactly the same Lie-Markov model as the previous case.

\medskip
\noindent
\underline{Semigroup 3}:  
\[
\begin{tabular}{c|cc}
    & $a_1$ & $a_2$ \\
    \hline
    $a_1$ & $a_1$ & $a_1$   \\
    $a_2$ & $a_2$ & $a_2$   
      \end{tabular}
\]
Taking the regular representation produces the matrices 
\[
A_1=
\left(
\begin{matrix}
1 & 1 \\
0 & 0
\end{matrix}
\right)
,\quad 
A_2=
\left(
\begin{matrix}
0 & 0 \\
1 & 1
\end{matrix}
\right).
\]
The resulting Markov model can equivalently be understood as producing the 2-state equal input model or the 2-state general Markov model (as discussed in Sec~\ref{sec:semi} and Sec~\ref{subsec:isodifferent}, respectively).
This model has full symmetry $\mathcal{S}_2$.

\medskip
\noindent
\underline{Semigroup 4}:
\[
\begin{tabular}{c|cc}
    & $a_1$ & $a_2$ \\
    \hline
    $a_1$ & $a_1$ & $a_2$   \\
    $a_2$ & $a_1$ & $a_2$   
      \end{tabular}
\]
This is the anti-isomorphic copy of the previous case.

Taking the regular representation produces the matrices 
\[
A_1=
\left(
\begin{matrix}
1 & 0 \\
0 & 1
\end{matrix}
\right)
,\quad 
A_2=
\left(
\begin{matrix}
1 & 0 \\
0 & 1
\end{matrix}
\right),
\]
which is not injective and produces the trivial Lie-Markov model with vanishing rate matrices $Q\!=\!0$.

\medskip
\noindent
\underline{Semigroup 5}: 
\[
\begin{tabular}{c|cc}
    & $a_1$ & $a_2$ \\
    \hline
    $a_1$ & $a_1$ & $a_2$   \\
    $a_2$ & $a_2$ & $a_1$   
      \end{tabular}
\]
This is the (abelian) cyclic group on two elements $C_2\cong \mathbb{Z}_2$.
Taking the regular representation produces the matrices 
\[
A_1=
\left(
\begin{matrix}
1 & 0 \\
0 & 1
\end{matrix}
\right)
,\quad 
A_2=
\left(
\begin{matrix}
0 & 1 \\
1 & 0
\end{matrix}
\right),
\]
which  produces the group-based model
\[
Q=
\alpha L_1+\beta L_2
=
\left(
\begin{matrix}
-\alpha & \phantom{-}\alpha \\
\phantom{-}\alpha & -\alpha
\end{matrix}
\right),
\]
also known as the `binary symmetric' model.
This is a one-dimensional Lie-Markov model with trivial commutator relations.
This model has full symmetry $\mathcal{S}_2$.

\subsection{Three-state semigroup models}

There are 24 non-isomorphic semigroups with degree 3.
The 24 cases split into 12 which are `self' anti-isomorphic plus 6 anti-isomorphic pairs.
Out the 24 possibilities, we found that there are \emph{two} semigroup based models which form non-reducible Markov chains.
We provide a complete set of results in the supplementary material and simply present the interesting cases here.

\medskip
\noindent
\underline{Example 1}: 
\[
\begin{tabular}{c|ccc}
    & $a_1$ & $a_2$ & $a_3$ \\
    \hline
    $a_1$ & $a_1$ & $a_1$ & $a_1$   \\
    $a_2$ & $a_2$ & $a_2$ & $a_2$  \\
    $a_3$ & $a_3$ & $a_3$ & $a_3$
      \end{tabular}
\]
Taking the regular representation produces the matrices 
\[
A_1=
\left(
\begin{matrix}
1 & 1 & 1 \\
0 & 0 & 0\\
0 & 0 & 0
\end{matrix}
\right)
,\quad 
A_2=
\left(
\begin{matrix}
0 & 0 & 0 \\
1 & 1 & 1\\
0 & 0 & 0
\end{matrix}
\right)
,\quad 
A_3=
\left(
\begin{matrix}
0 & 0 & 0 \\
0 & 0 & 0\\
1 & 1 & 1
\end{matrix}
\right),
\]
which  produces the three-state equal-input model:
\[
Q=
\alpha R_1+\beta R_2+\gamma R_3
=
\left(
\begin{matrix}
-\beta-\gamma & \alpha & \alpha \\
\beta & -\alpha-\gamma & \beta\\
\gamma & \gamma & -\alpha-\beta
\end{matrix}
\right),
\]
forming a matrix (and hence Lie) algebra with relations:
\[
R_iR_j=-R_j \implies \left[R_i,R_j\right]=R_i-R_j.
\]

This model has full symmetry $\mathcal{S}_3$.

\medskip
\noindent
\underline{Example 2}: 
\[
\begin{tabular}{c|ccc}
    & $a_1$ & $a_2$ & $a_3$ \\
    \hline
    $a_1$ & $a_1$ & $a_2$ & $a_3$   \\
    $a_2$ & $a_2$ & $a_3$ & $a_1$  \\
    $a_3$ & $a_3$ & $a_1$ & $a_2$
      \end{tabular}
\]

We recognise this as the cyclic group $C_3$ with generator $a_2$ satisfying $a_3=a_2^2$ and identity element $a_1=a_2^3$.
Taking the regular representation produces the permutation matrices 
\[
A_1=
\left(
\begin{matrix}
1 & 0 & 0 \\
0 & 1 & 0\\
0 & 0 & 1
\end{matrix}
\right)
,\quad 
A_2=
\left(
\begin{matrix}
0 & 0 & 1 \\
1 & 0 & 0\\
0 & 1 & 0
\end{matrix}
\right)
,\quad 
A_3=
\left(
\begin{matrix}
0 & 1 & 0 \\
0 & 0 & 1\\
1 & 0 & 0
\end{matrix}
\right).
\]
In this case, $L_1=-\id+A_1=0$, so we obtain the two-dimensional Lie-Markov model:
\[
Q=
\alpha L_2+\beta L_3
=
\left(
\begin{matrix}
-\alpha-\beta & \alpha & \beta \\
\beta & -\alpha-\beta & \alpha\\
\alpha & \beta & -\alpha-\beta
\end{matrix}
\right).
\]

We can also interpret this as the group-based model obtained from the group $\mathbb{Z}_3\cong C_3$ with matrix entries satisfying $Q_{ij}=f(i-j)$ where the row and column indices $i,j\in \mathbb{Z}_3$, $f(1)=\beta$ and $f(2)=\alpha$.

This model has full symmetry $\mathcal{S}_3$.

\medskip
\noindent
\underline{Example 3}:
Illustrating another example of a semigroup-based model with an absorbing state, consider the semigroup:
\[
\begin{tabular}{c|ccc}
    & $a_1$ & $a_2$ & $a_3$ \\
    \hline
    $a_1$ & $a_1$ & $a_1$ & $a_3$   \\
    $a_2$ & $a_2$ & $a_2$ & $a_3$  \\
    $a_3$ & $a_3$ & $a_3$ & $a_3$
      \end{tabular}
\]
Taking the regular representation produces the matrices:
\[
A_1=
\left(
\begin{matrix}
1 & 1 & 0 \\
0 & 0 & 0 \\
0 & 0 & 1
\end{matrix}
\right)
,\quad 
A_2=
\left(
\begin{matrix}
0 & 0 & 0 \\
1 & 1 & 0\\
0 & 0 & 1
\end{matrix}
\right)
,\quad 
A_3=
\left(
\begin{matrix}
0 & 0 & 0 \\
0 & 0 & 0\\
1 & 1 & 1
\end{matrix}
\right),
\]
and hence the semigroup-based model:
\[
Q=\alpha L_1+\beta L_2+\gamma L_3
=
\left(
\begin{matrix}
-\beta-\gamma & \alpha & 0 \\
\beta & -\alpha-\gamma & 0\\
\gamma & \gamma & 0
\end{matrix}
\right),
\]
for which the third state is an absorbing state.
The resulting Lie algebra has commutators:
\[
\left[L_1,L_2\right]=L_1-L_2,\quad \left[L_1,L_3\right]=0,\quad \left[L_2,L_3\right]=0.
\]

Due to the distinguished third state, this model has symmetry group $\{e,(12)\}<\mathcal{S}_3$.

\section{Discussion}

We did not attempt to derive semigroup-based models for any more than the $k\!=\!4$ state case.
One obstruction to going further is the combinatorial explosion in the number non-isomorphic semigroups of degree $k$ (OEIS sequence A027851): 
\[
1,5,24,188,1915,28634,1627672,3684030417,\ldots .
\]
Clearly it becomes computationally infeasible to systematically explore all possible cases as we have done in this paper.
One possibile method to proceed further, is to reject all those semigroups that do not satisfy a reasonable set of permutation symmetries.
For a semigroup $S$, the appropriate definition of a symmetry corresponds to the existence of an automorphism (self isomorphism) $\varphi: S\rightarrow S$.
It is clear that the reasonable semigroup-based models we have derived in this paper are based on semigroups with non-trivial automorphism groups.
However, the precise connection to the reasonableness of the so-derived Lie-Markov model is not clear and we leave this as a matter for future work.

Another reason we did not extend our constructions further beyond the $k\!=\!4$ case is due to our particular interest in phylogenetics and DNA substitution models.
Thus, although the primary focus in this paper was mathematical, it is worthwhile to ponder whether the new 4-state model we obtained is useful for the analysis of real data sets.
We also leave this for future work.

\subsubsection*{Acknowledgement}

This research was supported by Australian Research Council (ARC) Discovery Grant DP150100088. 

\bibliographystyle{plain}
\bibliography{biblio}

\end{document}